\documentclass[11pt, oneside]{article}   	
\usepackage{geometry}                		
\usepackage{graphicx}				
\usepackage{amssymb,amsthm}					
\usepackage{natbib}
\usepackage{amsmath}
\usepackage{epstopdf}
\usepackage{bbm}
\usepackage{mathtools}

\newcommand{\dotp}[2]{\left\langle #1, #2\right\rangle}

\newcommand{\abs}[1]{\left\vert#1\right\vert}

\newcommand{\norm}[1]{\big\Vert#1\big\Vert}

\newcommand\wt[1]{{ \widetilde{#1} }}

\newcommand \bbP{\mathbb{P}}
\newcommand \bbE{\mathbb{E}}

\def\m{\mathcal}
\def\mb{\mathbb}
\def\mr{\mathrm}

\def\ind{\mathbbm{1}}

\newtheorem{theorem}{Theorem}[section]
\newtheorem{lemma}[theorem]{Lemma}

\newtheorem{corollary}[theorem]{Corollary}
\newtheorem{rem}{Remark}[section]

\begin{document}

\title{Optimal Bayesian estimation in stochastic block models}

\author{ {\bf Debdeep Pati} \\Department of Statistics, Florida State University, Tallahassee, FL, \\email: debdeep@stat.fsu.edu \\
{\bf Anirban Bhattacharya}\\
Department of Statistics, Texas A \&M University, College Station, TX, \\email: anirbanb@stat.tamu.edu
}

\maketitle

\begin{abstract}
With the advent of structured data in the form of social networks, genetic circuits and protein interaction networks, statistical analysis of networks has gained popularity over recent years.  Stochastic block model constitutes a classical cluster-exhibiting random graph model for networks.  There is a  substantial amount of literature devoted to proposing strategies for estimating and inferring parameters of the model, both from classical and Bayesian viewpoints. Unlike the classical counterpart, there is however a dearth of theoretical results on the accuracy of estimation in the Bayesian setting. In this article, we undertake a theoretical investigation of the posterior distribution of the parameters in a stochastic block model. In particular, we show that one obtains optimal rates of posterior convergence with routinely used multinomial-Dirichlet priors on cluster indicators and uniform priors on the probabilities of the random edge indicators.  En route, we develop geometric embedding  techniques to exploit the lower dimensional structure of the parameter space which may be of independent interest. 
\end{abstract}

\section{Introduction}
Data available in the form of networks are increasingly becoming common in applications ranging from brain connectivity, protein interactions, web applications and social networks to name a few, motivating an explosion of activity in the statistical analysis of networks in recent years \cite{goldenberg2010survey}. Estimating large networks offers unique challenges in terms of structured dimension reduction and estimation in stylized domains, 
necessitating new tools for inference.  A rich variety of probabilistic models have been studied for network estimation, ranging from the classical Erdos and Renyi graphs 
\cite{erdos1961evolution}, exponential random graph models \cite{holland1981exponential}, stochastic block models \cite{holland1983stochastic}, markov graphs 
\citep{frank1986markov} and latent space models \citep{hoff2002latent} to name a few. 

In a network with $n$ nodes, there are $O(n^2)$ possible connections betweens pairs of nodes, the exact number depending on whether the network is directed/undirected and whether self-loops are permitted.  
A common goal of the parametric models mentioned previously is to parsimoniously represent the $O(n^2)$ probabilities of connections between pairs of nodes in terms of fewer parameters. The stochastic block model achieves this by clustering the nodes into $k \ll n$ groups, with the probability of an edge between two nodes solely dependent
on their cluster memberships. The block model originated in the mathematical sociology literature \cite{holland1983stochastic}, with subsequent widespread applications in statistics \cite{wang1987stochastic,snijders1997estimation,nowicki2001estimation}. In particular, the clustering property of block models offers a natural way to find {\em communities} within networks, inspiring a large literature on community detection \cite{bickel2009nonparametric,newman2012communities,zhao2012consistency,karrer2011stochastic,zhao2011community,amini2013pseudo}. Various modifications of the stochastic block model have also been proposed, including the mixed membership stochastic block model \cite{airoldi2009mixed} and degree-corrected stochastic block model \cite{zhao2012consistency}.  

Statistical accuracy of parameter inference in the stochastic block model is of growing interest, with one of the objects of interest being the $n \times n$ matrix of probabilities of edges between pairs of nodes, which we shall denote by $\theta = (\theta_{ij})$. Using a singular-value thresholding approach, \cite{chatterjee2014matrix} obtained a $\sqrt{k/n}$ rate for estimating $\theta$ with respect to the squared $\ell_2$ distance in a $k$-component stochastic block model. In a recent technical report, \cite{gao2014rate} obtained an improved $k^2/n^2 + \log k/n$ rate by considering a least-squares type estimator. They also showed that the resulting rate is minimax-optimal; 
interestingly the minimax rate comprises of two parts which \cite{gao2014rate} refer to as the {\em nonparametric} and {\em clustering} rates respectively. Among other related work, \cite{bickel2013asymptotic} provided conditions for asymptotic normality of maximum likelihood estimators in stochastic block models. 

In this article, we consider a Bayesian formulation of a stochastic block model where $\theta$ is equipped with a hierarchical prior and study the convergence of the posterior distribution assuming the data to be generated from a stochastic block model. We show that one obtains the minimax rate of posterior convergence with essentially automatic prior choices, such as multinomial-Dirichlet priors on cluster indicators and uniform priors on the probability of the random edge indicators. Such priors are commonly used and there is a sizable literature \cite{snijders1997estimation,nowicki2001estimation,golightly2005bayesian,mcdaid2013improved} on posterior sampling and inference in the stochastic block model. However, to the best of our knowledge, the present paper is the first to study the asymptotic properties of Bayesian estimation in stochastic block models. 

Theoretical investigation of the posterior distribution in block models offers some unique challenges relative to the small but growing literature on posterior convergence in high-dimensional {\em sparse} problems \cite{castilloneedles,pati2014posterior,banerjee2014posterior,castillo2015bayesian}. When a large subset of the parameters are exactly or approximately zero, the sparsity assumption can be exploited to reduce the complexity of the model space to derive tests for the true parameter versus the complement of a neighborhood of the true parameter \cite{castilloneedles,pati2014posterior}. It is now well appreciated that constructing such tests play a crucial role in posterior asymptotics \cite{ghosal2000convergence,gine2011rates}. In the present setting, we exploit the parsimonious structure of the parameters space as a result of clustering of $n$ nodes into $k < n$ communities and develop geometric embedding techniques to derive such tests.

\section{Preliminaries}\label{sec:prelim}
For $\m S \subset \mb R$, we shall denote the set of all $d \times d$ matrices with entries in $\m S$ by $\m S^{d \times d}$. For any $B = (B_{ll'}) \in \mb R^{d \times d}$, we denote the Euclidean (equivalently Frobenius) norm of $B$ by $\norm{B} = \sqrt{\sum_{l=1}^d \sum_{l'=1}^d B_{ll'}^2}$. Given $X^* \in \mb R^{d \times d}, W \in \mb R_+^{d \times d}$, let $\xi_{d^2}(X^*; W)$
denote the unit ellipsoid with {\em center} $X^*$ and {\em weight} $W$ given by
\begin{align}\label{eq:ellipsoid}
\xi_{d^2}(X^*; W) = \bigg\{ X \in \mb R^{d \times d} : \sum_{l=1}^d \sum_{l'=1}^d W_{l l'} (X_{l l'} - X^*_{ll'})^2 \le 1 \bigg\}.
\end{align}
Viewed as a subset of $\mb R^{d^2}$, the Euclidean volume of $\xi_{d^2}(X^*; W)$, denoted by $| \xi_{d^2}(X^*; W) |$, is
\begin{align}\label{eq:vol}
| \xi_{d^2}(X^*; W) | = \frac{\pi^{d^2}}{\Gamma(d^2/2 + 1)} \prod_{l=1}^d \prod_{l' = 1}^d W_{ll'}^{-1/2}.
\end{align} 

Given sequences $\{a_n\}, \{b_n\}$, $a_n \lesssim b_n$ indicates there exists a constant $K > 0$ such that $a_n \le K b_n$ for all large $n$. We say $a_n \asymp b_n$ when $a_n \lesssim b_n$ and $b_n \lesssim a_n$. Throughout, $C, C'$ denote positive constants whose values might change from one line to the next. 

\section{Stochastic Block models} \label{sec:sbm}

Let $A  = (A_{ij}) \in \{0, 1\}^{n \times n}$ denote the adjacency matrix of a network with $n$ nodes, with 
$A_{ij} = 1$ indicating the presence of an edge from node $i$ to node $j$ and $A_{ij} = 0$ indicating a lack thereof. To keep the subsequent notation clean, we shall consider directed networks with self-loops so that $A_{ij}$ and $A_{ji}$ need not be the same and $A_{ii}$ can be both $0$ and $1$. Our theoretical results can be trivially modified to undirected networks with or without self-loops. 

Let $\theta_{ij}$ denote the probability of an edge from node $i$ to $j$, with $A_{ij} \sim \mbox{Bernoulli}(\theta_{ij})$ independently for $1 \le i, j \le n$. A stochastic block model postulates that the nodes are clustered into communities, with the probability of an edge between two nodes solely dependent on their community memberships. Specifically, let $z_i \in \{1, \ldots, k\}$ denote the cluster membership of the $i$th node and $Q = (Q_{rs}) \in [0, 1]^{k \times k}$ be a matrix of probabilities, with $Q_{rs}$ indicating the probability of an edge from any node $i$ in cluster $r$ to any node $j$ in cluster $s$. With these notations, a $k$-component stochastic block model is given by 
\begin{align}\label{eq:sBm}
A_{ij} \sim \mbox{Bernoulli}(\theta_{ij}), \quad \theta_{ij} = Q_{z_i z_j}.
\end{align}
We use $\bbE_{\theta} \slash \bbP_{\theta}$ to denote an expectation/probability under the sampling mechanism \eqref{eq:sBm}. 

The stochastic block model clearly imposes a parsimonious structure on the node probabilities $\theta = (\theta_{ij})$ when $k \ll n$, reducing the effective number of parameters from $O(n^2)$ to $O(k^2 + n)$. To describe the parameter space for $\theta$, we need to introduce some notations. 
For $k \le n$, let $\m Z_{n, k} = \big\{(z_1, \ldots, z_n) : z_i \in \{1, \ldots, k\}, 1 \le i \le n \big\}$ denote all possible clusterings of $n$ nodes into $k$ clusters. Elements of $\m Z_{n, k}$ with be denoted by $z = (z_1, \ldots, z_n)$. For any $1 \le r \le k$, $z^{-1}(r)$ is used as a shorthand for $\{1 \le i \le n : z_i = r\}$; the nodes belonging to cluster $r$. When $z$ is clear from the context, we shall use $n_r = |z^{-1}(r)|$ to denote the number of nodes in cluster $r$; clearly $\sum_{r=1}^k n_r = n$. 
With these notations, the parameter space $\Theta_k$ for $\theta$ is given by 
\begin{align}\label{eq:theta_k}
\Theta_k =  \{ \theta \in [0,1]^{n\times n} : \theta_{ij} = Q_{z_i z_j}, \, \, z \in \m Z_{n, k}, \, 
Q \in [0, 1]^{k \times k} \}.
\end{align}
For any $z \in \m Z_{n, k}$ and $Q \in [0, 1]^{k \times k}$, we denote the corresponding $\theta \in \Theta_k$ by $\theta^{z, Q}$, so that $\theta^{z, Q}_{ij} = Q_{z_i z_j}$. In fact, $(z, Q) \mapsto \theta^{z, Q}$ is a surjective map from $\m Z_{n, k} \times [0, 1]^{k \times k} \to \Theta_k$, though it is clearly not injective.

Given $z \in \m Z_{n, k}$, let $A_{[rs]}$ denote the $n_r \times n_s$ sub matrix of $A$ consisting of entries $A_{ij}$ with $z_i = r$ and $z_j = s$. The joint likelihood of $A$ under model \eqref{eq:sBm} can be expressed as 
{\small \begin{align}\label{eq:like}
P(A \mid z, Q) = \prod_{r = 1}^k \prod_{s = 1}^k P(A_{[rs]} \mid z, Q), \quad P(A_{[rs]} \mid z, Q) = \prod_{i: z_i = r} \prod_{j: z_j = s} Q_{rs}^{A_{ij}} (1 - Q_{rs})^{1 - A_{ij}}.
\end{align}}
A Bayesian specification of the stochastic block model can be completed by assigning independent priors to $z$ and $Q$, which in turn induces a prior on $\Theta_k$ via the mapping $(z, Q) \mapsto \theta^{z, Q}$. We generically use $p(z, Q) = p(z) p(Q)$ to denote the joint prior on $z$ and $Q$. The induced prior on $\Theta_k$ will be denoted by $\Pi(\theta)$ and the corresponding posterior given data $A = (A_{ij})$ will be denoted by $\Pi_n(\theta \mid A)$. The following fact is useful and heavily used in the sequel: for any $U \subset \Theta_k$, 
\begin{align}\label{eq:prior_U}
\Pi(U) = \sum_{z \in \m Z_{n,k}} \Pi(U \mid z) \ p(z) = \sum_{z \in \m Z_{n,k}} p(Q : \theta^{z,Q} \in U) \ p(z),
\end{align}
where the second equality uses the independence of $z$ and $Q$. Specific choices of $p(z)$ and $p(Q)$ are discussed below. 

We assume independent $U(0, 1)$ prior on $Q_{rs}$. We consider a hierarchical prior on $z$ where each node has probability $\pi_r$ of being allocated to the $r$th cluster independently of the other nodes, and the vector of probabilities $\pi = (\pi_1, \ldots, \pi_k)$ follows a $\mbox{Dirichlet}(\alpha_1, \ldots, \alpha_k)$ prior. Here $\alpha_1, \ldots, \alpha_k$ are fixed hyper-parameters that do not depend on $k$ or $n$; a default choice is $\alpha_r = 1/2$ for all $r = 1, \ldots, k$. Model \eqref{eq:sBm} along with the prior specified above can be expressed hierarchically as follows:
\begin{align}
& Q_{rs} \stackrel{\text{ind}} \sim U(0, 1),  \quad r, s = 1, \ldots, k, \label{eq:sbm1}\\
& P(z_i = k \mid \pi) = \pi_k, \quad i = 1, \ldots, n,  \label{eq:sbm2}\\
& \pi \sim \mbox{Dirichlet}(\alpha_1, \ldots, \alpha_k), \label{eq:sbm3}\\
& A_{ij} \mid z, Q  \stackrel{\text{ind}} \sim \mbox{Bernoulli}(\theta_{ij}), \quad \theta_{ij} = Q_{z_i z_j}. \label{eq:sbm4}
\end{align}
A hierarchical specification as in (or very similar to) \eqref{eq:sbm1} -- \eqref{eq:sbm4} has been commonly used in the literature; see for example, \cite{snijders1997estimation,nowicki2001estimation,golightly2005bayesian,mcdaid2013improved}. Analytic marginalizations can be carried out due to the conjugate nature of the prior, facilitating posterior sampling \cite{mcdaid2013improved}. In particular, 
using standard multinomial-Dirichlet conjugacy, the marginal prior of $z$ can be written as 
\begin{align}\label{eq:dir_mult}
p(z) = \frac{\Gamma(\sum_{r=1}^k \alpha_r)}{\Gamma(n + \sum_{r=1}^k \alpha_r)} \prod_{r=1}^k \frac{\Gamma(n_r + \alpha_r)}{\Gamma(\alpha_r)}, \quad z \in \m Z_{n, k},
\end{align}
where recall that $n_r = \sum_{i=1}^n \ind(z_i = r)$. The following lemma provides a bound on the ratio of prior probabilities $p(z)/p(z')$ which is used subsequently in the proof of our main theorem. 
\begin{lemma}\label{lem:prior_rat}
Assume $z' \in \m Z_{n,k}$ with $n'_r = \sum_{i=1}^n \ind(z_i' = r) \ge 1$ for all $r = 1, \ldots k$. Then, $\max_{z \in \m Z_{n,k}} p(z)/p(z') \le e^{C n \log k}$, where $C$ is a positive constant. 
\end{lemma}
\begin{proof}
Fix $z \in \m Z_{n,k}$. From \eqref{eq:dir_mult}, $p(z)/p(z') = \prod_{r=1}^k \Gamma(n_r + \alpha_r)/\Gamma(n'_r + \alpha_r)$. Define non-negative integers $\beta_r =  \lceil \alpha_r \rceil, \gamma_r =  \lfloor \alpha_r \rfloor$, $\beta_{\cdot} = \sum_{r=1}^k \beta_r, \gamma_{\cdot} = \sum_{r=1}^k \gamma_r$. Recall the following facts about the Gamma function: (i) $\Gamma(x)$ is decreasing on $(0, 1)$ with $1/(2x) \le \Gamma(x) \le 1/x$, (ii) $\Gamma(x) \le 1$ for $x \in [1, 2]$, $\Gamma(1)= \Gamma(2) =1$ and (iii) $\Gamma(x)$ is increasing for $x \ge 2$. First, we claim $\Gamma(n_r' + \alpha_r) \ge C' \Gamma(n_r' + \gamma_r)$ for all $r$ for an absolute constant $C' > 0$. To see this, recalling that $n_r' \ge 1$, separately consider the cases (a) $n_r' \ge 2$, (b) $n_r' = 1$ and $\alpha_r > 1$ and (c) $n_r' = 1$ and $\alpha_r < 1$. Cases (a) and (b) follow from fact (iii) above with $C' = 1$. For case (c), $\Gamma(n_r' + \alpha_r) = \Gamma(1 + \alpha_r) = \alpha_r \Gamma(\alpha_r) \ge 1/2$ by fact (i) and $\Gamma(n_r' + \gamma_r) = \Gamma(1) = 1$; therefore one may choose $C' = 1/2$ here.  Next, we claim that $\Gamma(n_r + \alpha_r) \le C \Gamma(n_r + \beta_r)$ for all $r$ and an absolute constant $C > 0$. To see this, separately consider cases (a) $n_r \ge 1$, (b) $n_r = 0$ and $\alpha_r \ge 1$ and (c) $n_r = 0$ and $\alpha_r < 1$. Cases (a) and (b) once again follow from fact (iii) above with $C = 1$. For case (c), $\Gamma(n_r + \alpha_r) = \Gamma(\alpha_r)$ and $\Gamma(n_r + \beta_r) = \Gamma(1) = 1$, so one may choose $C = \max_{1 \le r \le k} \{1/\Gamma(\alpha_r)\}$. We thus have
\begin{align}\label{eq:rat_1}
\frac{p(z)}{p(z')} & = \prod_{r=1}^k \frac{\Gamma(n_r + \alpha_r)}{\Gamma(n_r' + \alpha_r)} \le \bigg(\frac{C}{C'}\bigg)^k \prod_{r=1}^k  \frac{\Gamma(n_r + \beta_r)}{\Gamma(n_r' + \gamma_r)} \nonumber \\
&= \bigg(\frac{C}{C'}\bigg)^k \, \frac{(n + \beta_{\cdot} - k)!}{(n + \gamma_{\cdot} - k)!}  \,  \frac{ {n +\gamma_{\cdot} - k  \choose
 n_1' + \gamma_1 -1, \ldots, n_k'+ \gamma_k -1} }{  {n +\beta_{\cdot} - k  \choose
 n_{1} + \beta_1 -1, \ldots, n_{k}+ \beta_k -1} },
\end{align}
where we used $\sum_{r=1}^k n_r = \sum_{r=1}^k n_r' = n$ and ${n \choose m_1, \ldots, m_k} := n!/(m_1 ! \ldots m_k !) \ind(m_1 + \ldots m_k = n)$ is the multinomial coefficient. 
Since the multinomial coefficient as a function of $m_1, \ldots, m_{k-1}$, $(m_1, \ldots, m_{k-1}) \mapsto {m \choose m_1, \ldots, m_k}$, attains
 its minimum if $m_r = m$ for some $r$ and $m_l =0$ for $l \neq r$, and maximum if $m_r = m/k$ for all $r = 1, \ldots, k$, we have from \eqref{eq:rat_1},
\begin{align*}
\frac{p(z)}{p(z')} \le  \bigg(\frac{C}{C'}\bigg)^k \, \frac{(n + \beta_{\cdot} - k)!}{(n + \gamma_{\cdot} - k)!}  \frac{(n+ \gamma_{\cdot} - k)!}{[\{ (n  + \beta_{\cdot} - k )/k\}!]^k}  = \bigg(\frac{C}{C'}\bigg)^k  \frac{(n + \beta_{\cdot} - k)!}{[\{ (n  + \beta_{\cdot} - k )/k\}!]^k}.
\end{align*}
Set $m = n + \gamma_{\cdot} - k$. Since $m! \asymp (m/e)^m \sqrt{m}$ by Stirling's bound, $m!/\{(m/k)!\}^k \lesssim k^m \sqrt{m} (k/m)^{k/2} \lesssim k^{C n}$ for $C > 1$ large enough. The conclusion follows. 

\end{proof}

\section{Posterior convergence rates in Stochastic Block Models}
We are interested in concentration properties of the posterior $\Pi_n(\cdot \mid A)$ assuming the true data-generating parameter $\theta^0 \in \Theta_k$. To measure the discrepancy in the estimation of $\theta^0 \in \Theta_k$, the mean squared error has been used in the frequentist literature,
\begin{align}\label{eq:diet}
\frac{1}{n^2} \sum_{i=1}^n \sum_{j=1}^n (\hat{\theta}_{ij} - \theta_{ij}^0)^2 =  \frac{1}{n^2}
 \norm{\hat{\theta} - \theta^0}^2,
\end{align}
where $\hat{\theta}$ is an estimator of $\theta^0$.  \cite{chatterjee2014matrix} proposed estimating $\theta^0$ using a low rank decomposition of the adjacency matrix $A$ followed by a singular value decomposition to obtain a convergence rate of $\sqrt{k/n}$. More recently, \cite{gao2014rate} considered a least squares type approach which can be related to maximum likelihood estimation where the Bernoulli likelihood is replaced by a Gaussian likelihood. They obtained a rate of $k^2/n^2 + \log k/n$, which they additionally showed to be the minimax rate over $\Theta_k$, i.e.,
\begin{align}\label{eq:minmax}
\inf_{\hat{\theta}} \sup_{\theta^0 \in \Theta_k} \bbE_{\theta_0} \frac{1}{n^2}
 \norm{\hat{\theta} - \theta^0}^2 \asymp \frac{k^2}{n^2} + \frac{\log k}{n}.
\end{align} 
Interestingly, the minimax rate has two components, $k^2/n^2$ and $\log k/n$. \cite{gao2014rate} refer to the $k^2/n^2$ term in the minimax rate as the {\em nonparametric rate}, since it arises from the need to estimate $k^2$ unknown elements in $Q$ from $n^2$ observations. The second part, $\log k/n$, is termed as the {\em clustering rate}, which appears since the clustering configuration $z$ is unknown and needs to be estimated from the data. Observe that the clustering rate grows logarithmically in $k$. Parameterizing $k = n^{\delta}$ with $\delta \in [0, 1]$, the interplay between the two components becomes clearer (refer to equation 2.6 of \cite{gao2014rate}); in particular, the clustering rate dominates when $k$ is small and the nonparametric rate dominates when $k$ is large. 

To evaluate Bayesian procedures from a frequentist standpoint, one seeks for the minimum possible sequence $\epsilon_n \to 0$ such that the posterior probability assigned to the complement of an $\epsilon_n$-neighborhood (blown up by a constant factor) of $\theta^0$ receives vanishingly small probabilities. The smallest such $\epsilon_n$ is called the {\em posterior convergence rate} \cite{ghosal2000convergence}. There is now a growing body of literature showing that Bayesian procedures achieve the frequentist minimax rate of posterior contraction (up to a logarithmic term) in  models where the parameter dimension 
grows with the sample size; see \cite{bontemps2011bernstein,castilloneedles,pati2014posterior,banerjee2014posterior,van2014horseshoe,castillo2015bayesian} for some flavor of the recent literature. 



We now state the main result of this article where we derive the convergence rate of the posterior arising from the hierarchical formulation \eqref{eq:sbm1} -- \eqref{eq:sbm4}. 
\begin{theorem} \label{thm:sbm_nonadaptive}
Assume $A = (A_{ij})$ is generated from a $k$-component stochastic block model \eqref{eq:sBm} with the true data-generating parameter $\theta^0 = (\theta^0_{ij}) \in \Theta_k$, where $\Theta_k$ is as in \eqref{eq:theta_k}. Further assume that there exists a small constant $\delta \in (0, 1)$ such that $ \theta^0_{ij} \in(\delta,  1 - \delta)$ for all $i, j = 1, \ldots, n$. Suppose the hierarchical Bayesian model \eqref{eq:sbm1} -- \eqref{eq:sbm4} is fitted. Then, with $\epsilon_n^2 = k^2\log(n/k)/ n^2  + \log k/n$ and a sufficiently large constant $M > 0$, 
\begin{align}\label{eq:post_conc}
\lim_{n \to \infty} \bbE_{\theta^0}  \Pi_n \bigg \{ \frac{1}{n^2} \norm{\theta - \theta^0}^2 > M^2 \epsilon_n^2  \mid A \bigg \} = 0.
\end{align}
\end{theorem}
\begin{rem}
Since $\theta^0 \in \Theta_k$, following the discussion after \eqref{eq:theta_k}, there exists $z^0 \in \m Z_{n, k}$ and $Q^0 \in [0, 1]^{k \times k}$ such that $\theta^0 = \theta^{z^0, Q^0}$. The condition of the theorem posits that all entries of $Q^0$ lie in $(\delta, 1 - \delta)$. The assumption $\theta^0 \in \Theta_k$ also implicitly implies that all the clusters have at least one observation, i.e., $n^0_r = \sum_{i=1}^n \ind(z^0_i = r) \ge 1$ for all $r = 1, \ldots, k$; otherwise $\theta^0 \in \Theta_l$ for some $l < k$.
\end{rem}
A proof of Theorem \ref{thm:sbm_nonadaptive} can be found in Section 5. Theorem \ref{thm:sbm_nonadaptive} shows that the posterior contracts at a (near) minimax rate of $k^2\log(n/k)/n^2 + \log k/n$. The nonparametric component of the rate is slightly hurt by a logarithmic term; appearance of such an additional logarithmic term is common in Bayesian nonparametrics. An inspection of the proof additionally reveals that the techniques can be trivially extended to undirected graphs with or without self-loops and will produce the same rate.

%
\subsection{Geometry of $\Theta_k$}

In this section, we derive a number of auxiliary results aimed at understanding the geometry of the parameter space $\Theta_k$. These results are used to prove Theorem \ref{thm:sbm_nonadaptive} and can be possibly of independent interest. 

We first state a testing lemma which harnesses the ability of the likelihood to separate points in the parameter space. 
\begin{lemma}\label{lem:test}
Assume $\theta^0 \ne \theta^1 \in \Theta_k$ and let $E = \{\theta \in [0, 1]^{n \times n} : \norm{\theta - \theta^0} \le \norm{\theta^1 - \theta^0}/2\}$ be an Euclidean ball of radius $\norm{\theta^1 - \theta^0}/2$ around $\theta^1$ inside $[0, 1]^{n \times n}$. Based on $A_{ij} \stackrel{\text{ind}}\sim \mbox{Bernoulli}(\theta_{ij})$ for $i, j = 1, \ldots, n$, consider testing 
$H_0: \theta = \theta^0 \, \mr{versus} \, H_1: \theta \in E$.
There exists a test function $\Phi$ such that 
\begin{align}\label{eq:pt_vs_pt}
\bbE_{\theta^0} (\Phi)  \le \exp\{- C_1 \norm{\theta^1 - \theta^0}^2 \},  \quad 
\sup_{\theta \in E } \bbE_{\theta} (1- \Phi)  \le \exp\{- C_2 \norm{\theta^1 - \theta^0}^2 \},
\end{align}
for constants $C_1, C_2 > 0$ independent of $n, \theta^1$ and $\theta^0$. 
\end{lemma}
\begin{proof}
Define the test function $\Phi$ as 
$$
\Phi =  \ind \bigg\{ \sum_{i=1}^n \sum_{j=1}^n ( \theta^1_{ij}  -  \theta^0_{ij} ) (A_{ij} - \theta^0_{ij})  >  \norm{\theta^1 - \theta^0}^2/4 \bigg\},
$$
where $\ind(\cdot)$ denotes the indicator of a set. We show below that this test has the desired error rates \eqref{eq:pt_vs_pt}.

We first bound the type-I error $\bbE_{\theta^0} (\Phi)$. Noting that under $\bbP_{\theta^0}$, $(A_{ij} - \theta^0_{ij})$ are independent zero mean random variables with $|A_{ij} - \theta^0_{ij}| < 1$, we use a version of 
Hoeffding's inequality (refer to Proposition 5.10 of \cite{vershynin2010introduction}) to conclude that,  
\begin{align*}
\bbE_{\theta^0} (\Phi) &= \bbP_{\theta^0} \bigg\{ \sum_{i=1}^n \sum_{j=1}^n ( \theta^1_{ij}  -  \theta^0_{ij} ) (A_{ij} - \theta^0_{ij})   >  \norm{\theta^1 - \theta^0}^2/4 \bigg\} \\
& \le  \exp \bigg\{-C_1 \frac{\norm{\theta^1 - \theta^0}^4}{\norm{\theta^1 - \theta^0}^2}\bigg\} = \exp\big\{-C_1 \norm{\theta^1 - \theta^0}^2 \big\}
\end{align*}
for a constant $C_1 > 0$ independent of $n, \theta^1$ and $\theta^0$. 

We next bound the type-II error $\sup_{\theta \in E} \bbE_{\theta}(1 - \Phi)$. Fix $\theta \in E$. We have,
\begin{align}
\bbE_{\theta}(1 - \Phi) 
& = \bbP_{\theta} \bigg\{\sum_{i=1}^n \sum_{j=1}^n ( \theta^1_{ij}  -  \theta^0_{ij} ) (A_{ij} - \theta^0_{ij})   < \norm{\theta^1 - \theta^0}^2/4 \bigg\} \notag \\
& = \bbP_{\theta} \bigg\{\sum_{i=1}^n \sum_{j=1}^n ( \theta^1_{ij}  -  \theta^0_{ij} ) (A_{ij} - \theta_{ij})   < \norm{\theta^1 - \theta^0}^2/4 -  \dotp{\theta^1 - \theta^0}{\theta - \theta^0}  \bigg\}, \label{eq:type_tw}
\end{align}
where we abbreviate $\dotp{\theta'}{\theta''} = \sum_{i=1}^n \sum_{j=1}^n \theta'_{ij} \theta''_{ij}$. Bound
\begin{align*}
&\dotp{\theta^1 - \theta^0}{\theta - \theta^0} \\
& = \dotp{\theta^1 - \theta^0}{\theta^1 - \theta^0} - \dotp{\theta^1 - \theta^0}{\theta^1 - \theta} \\
& \ge \norm{\theta^1 - \theta^0}^2 - \norm{\theta^1 - \theta^0}^2/2 = \norm{\theta^1 - \theta^0}^2/2,
\end{align*}
where the penultimate step used the Cauchy--Schwartz inequality along with the fact that $\norm{\theta - \theta^1} \le \norm{\theta^1 - \theta^0}/2$. Substituting in \eqref{eq:type_tw} and noting that under $\bbP_{\theta}$, $(A_{ij} - \theta_{ij})$ are independent zero mean bounded random variables, another application of Hoeffding's inequality yields
\begin{align*}
\bbE_{\theta} (1 - \Phi) & \le \bbP_{\theta} \bigg\{\sum_{i=1}^n \sum_{j=1}^n ( \theta^1_{ij}  -  \theta^0_{ij} ) (A_{ij} - \theta_{ij})   < - \norm{\theta^1 - \theta^0}^2/4 \bigg\} \\
& \le \exp \bigg\{-C_2 \frac{\norm{\theta^1 - \theta^0}^4}{\norm{\theta^1 - \theta^0}^2}\bigg\} = \exp\big\{-C_2 \norm{\theta^1 - \theta^0}^2 \big\} 
\end{align*}
for some constant $C_2 > 0$ independent of $n$ and $\theta$. Taking a supremum over $\theta \in E$ yields the desired result.
\end{proof}

Our next result is concerned with the structure of a specific type of Euclidean balls inside $\Theta_k$. Recall that $\theta^{z, Q}$ denotes the element of $\Theta_k$ with $\theta^{z, Q}_{ij} = Q_{z_i z_j}$. For $z \in \m Z_{n, k}$, let 
\begin{align}\label{eq:theta_kz}
\Theta_k(z) = \big\{ \theta^{z, Q} : Q \in [0, 1]^{k \times k} \big\}
\end{align}
denote a slice of $\Theta_k$ along $z$. In other words, given $z$, $\Theta_k(z)$ is the image of the map $Q \mapsto \theta^{z, Q}$ in $\Theta_k$. Suppose $\theta^* = \theta^{z^*, Q^*} \in \Theta_k$, and   consider a ball $B(z)$ in $\Theta_k(z)$ centered at $\theta^*$ of the form $B(z) = \big\{\theta \in \Theta_k(z) : \norm{\theta - \theta^*} < t \big\}$ for some $t > 0$. If $z^* = z$, then it is straightforward to observe that
\begin{align}\label{eq:reduct_1}
\norm{\theta^{z,Q} - \theta^{z^*, Q^*}}^2 = \sum_{r=1}^k \sum_{s=1}^k n_r n_s (Q_{rs} - Q_{rs}^*)^2,
\end{align}
where recall that $n_r = \sum_{i=1}^n \ind(z_i = r)$ for $r = 1, \ldots, k$. Therefore, although a subset of $[0, 1]^{n \times n}$, $B(z)$ can be identified with a $k^2$-dimensional ellipsoid in $[0,1]^{k \times k}$. When $z^* \ne z$, one no longer has a nice identity as above and the geometry of $B(z)$ is more difficult to describe. However, we show below in Lemma \ref{lem:ball} that $B(z)$ is always contained inside a set $\wt{B}(z)$ in $\Theta_k(z)$ which can be identified with a $k^2$-dimensional ellipsoid in $[0,1]^{k \times k}$. Recall our convention for describing ellipsoids from \eqref{eq:ellipsoid}. 
\begin{lemma}\label{lem:ball}
Fix $z^* \in \m Z_{n, k}, Q^* \in [0, 1]^{k \times k}$, and let $\theta^* = \theta^{z^*, Q^*}$. For $z \in \m Z_{n, k}$ and $t > 0$, let $B(z) = \big\{\theta \in \Theta_k(z) : \norm{\theta - \theta^*} < t \big\}$. 
Set $W_{rs} = n_r n_s/t^2$ and $W = (W_{rs})$, where $n_r = \sum_{i=1}^n \ind(z_i = r)$ for $r = 1, \ldots, k$. Then, $B(z) \subseteq \wt{B}(z)$, where 
\begin{align}
\wt{B}(z) = \big\{ \theta^{z, Q}: Q \in \xi_{k^2}(\bar{Q}^*, W) \, \cap \, [0, 1]^{k \times k} \big\}
\end{align}
for some $\bar{Q}^* \in [0, 1]^{k \times k}$ depending on $Q^*, z^*$ and $z$. 
In particular, if $z^* = z$, then $\bar{Q}^* = Q^*$ and the containment becomes equality, i.e., $B(z) = \wt{B}(z)$.
\end{lemma}
\begin{rem}\label{rem:ellip}
From \eqref{eq:ellipsoid}, $\xi_{k^2}(\bar{Q}^*, W)$ in Lemma \ref{lem:ball} is the collection of all $Q$ satisfying $\sum_{r=1}^k \sum_{s=1}^k n_r n_s(Q_{rs} - \bar{Q}_{rs}^*)^2 < t^2$. The last part of Lemma \ref{lem:ball} is consistent with the discussion preceding \eqref{eq:reduct_1}. When $z^* = z$, \eqref{eq:reduct_1} implies that $B(z)$ consists of all $\theta^{z, Q}$ with $Q \in [0, 1]^{k \times k}$ satisfying $\sum_{r=1}^k \sum_{s=1}^k n_r n_s (Q_{rs} - Q_{rs}^*)^2 < t^2$. 
\end{rem}
\begin{proof}
We begin by constructing $\bar{Q}^*$. For $1 \le r, r' \le k$, let $I_ {r, r'} =  z^{-1}(r) \cap  (z^*)^{-1}(r')$ and $n_{r, r'} = |I_{r, r'}|$. Clearly $\{ I_ {r, r'}, r' = 1, \ldots, k \}$ is a partition of $z^{-1}(r)$ and hence $n_r = \sum_{r'=1}^k n_{r, r'}$. With these notations, define 
\begin{align}\label{eq:q_starbar}
\bar{Q}_{rs}^* = \frac{1}{n_r n_s} \sum_{r' = 1}^k \sum_{s'=1}^k n_{r,r'} n_{s,s'} Q_{r',s'}^*.
\end{align}
Clearly, $\bar{Q}_{rs}^*$ is a weighted average of $Q_{r',s'}^*$ with weights proportional to $n_{r,r'} n_{s,s'}$ and therefore $\bar{Q}^* \in [0, 1]^{k \times k}$. For $\theta^{z, Q} \in \Theta_k(z)$, we have
$$
\norm{\theta^{z,Q} - \theta^{z^*,Q^*}}^2 = \sum_{i=1}^n \sum_{j=1}^n \big(Q_{z_i z_j} - Q^*_{z_i^* z_j^*} \big)^2.
$$
Expanding the squares, the term $\sum_{i=1}^n \sum_{j=1}^n Q_{z_i z_j}^2 = \sum_{r=1}^k \sum_{s=1}^k n_r n_s Q_{rs}^2$. 
The cross product term can be simplified to 
\begin{align}\label{eq:cp}
\sum_{i=1}^n \sum_{j=1}^n Q_{z_i z_j} Q^*_{z_i^* z_j^*} &= \sum_{r=1}^k \sum_{s=1}^k \sum_{r'=1}^k \sum_{s'=1}^k n_{r,r'} n_{s,s'} Q_{rs} Q_{r's'} \notag \\
& = \sum_{r=1}^k \sum_{s=1}^k n_r n_s Q_{rs} \bar{Q}^*_{rs}.
\end{align}
In view of these identities, we can write
\begin{align}
& \sum_{i=1}^n \sum_{j=1}^n \big(Q_{z_i z_j} - Q^*_{z_i^* z_j^*} \big)^2 
= \sum_{r=1}^k \sum_{s=1}^k n_r n_s (Q_{rs} - \bar{Q}^*_{rs})^2  + \notag \\
& \sum_{i=1}^n \sum_{j=1}^n ( Q^*_{z_i^* z_j^*} )^2  - \sum_{r=1}^k \sum_{s=1}^k n_r n_s (\bar{Q}^*_{rs})^2. \label{eq:differ}
\end{align}
We shall show below that the expression in \eqref{eq:differ} is non-negative. This completes the proof, since we then have
$\sum_{r=1}^k \sum_{s=1}^k n_r n_s (Q_{rs} - \bar{Q}^*_{rs})^2 \le \norm{\theta^{z,Q} - \theta^{z^*,Q^*}}^2 < t^2$ (see Remark \ref{rem:ellip}). Recalling the definition of $\bar{Q}^*_{rs}$ from \eqref{eq:q_starbar}, the expression in \eqref{eq:differ} can be written as
\begin{align*}
& \sum_{i=1}^n \sum_{j=1}^n ( Q^*_{z_i^* z_j^*} )^2  - \sum_{r=1}^k \sum_{s=1}^k n_r n_s (\bar{Q}^*_{rs})^2 \\
= & \sum_{r=1}^k \sum_{s=1}^k \bigg[\sum_{r'=1}^k \sum_{s'=1}^k n_{r,r'} n_{s,s'} (Q^*_{r's'})^2  - \frac{1}{n_r n_s} \bigg\{\sum_{r'=1}^k \sum_{s'=1}^k n_{r,r'} n_{s,s'} Q^*_{r's'} \bigg\}^2 \bigg].
\end{align*}
The non-negativity of this quantity now follows from the Cauchy--Schwartz inequality $(\sum_{\alpha} a_{\alpha} b_{\alpha})^2 \le \sum_{\alpha} a_{\alpha}^2 b_{\alpha} \sum_{\alpha} b_{\alpha}$, with $a_{\alpha} = Q^*_{r's'}$ and $b_{\alpha} = n_{r,r'} n_{s,s'}$; and the fact that $\sum_{r'=1}^k n_{r,r'} = n_r$.

When $z^* = z$, it is clear that $n_{r,r'} = \delta_{rr'}$ and hence $\bar{Q}^* = Q^*$. All subsequent inequalities then become equalities and the last part is proved. 

\end{proof}

\begin{corollary}\label{cor:ellip1}
Inspecting the proof of Lemma \ref{lem:ball}, the condition $Q \in [0, 1]^{k \times k}$ is only used to show that $bar{Q}^* \in [0, 1]^{k \times k}$. If we let $Q$ to be unrestricted, then the containment relation continues to hold as subsets of $\mb R^{k \times k}$, i.e.,  
\begin{align}
\bigg\{\theta^{z, Q} : Q \in \mb R^{k \times k}, \norm{\theta^{z, Q} - \theta^{z^*, Q^*}} < t \bigg\} \subseteq \bigg\{ \theta^{z, Q} : Q \in \xi_{k^2}(\bar{Q}^*, W) \bigg\}, 
\end{align}
with equality when $z^* = z$.
\end{corollary}
Lemma \ref{lem:ball} and Corollary \ref{cor:ellip1} crucially exploit the lower dimensional structure underlying the parameter space $\Theta_k$ and is used subsequently multiple times. First, recall from \eqref{eq:prior_U} that one needs a handle on $p(Q : \theta^{z,Q} \in U)$ to bound the prior probability of $U \subset \Theta_k$. In particular, if $U = \{\norm{\theta - \theta^0} < t\}$, then $p(Q : \theta^{z,Q} \in U)$ equals the volume of $U \cap \Theta_k(z)$, which can be suitably bounded by the volume of the bounding $k^2$ dimensional ellipsoid. 
Second, a handle on the size of balls in $\Theta_k$ facilitates calculating the complexity of the model space (in terms of metric entropy) which is pivotal in proving the posterior concentration; in particular, to extend the test function in Lemma \ref{lem:test} to construct test functions against more complex alternatives in Lemma \ref{lem:test_ann} below. Once again, the dimensionality reduction is key to preventing the metric entropy from growing too fast.

\begin{lemma}\label{lem:test_ann}
Recall $\epsilon_n$ from Theorem \ref{thm:sbm_nonadaptive}. Assume $\theta^0 \in \Theta_k$ and for $l \ge 1$, let $U_{l, n} = \big\{ \theta \in \Theta_k : l n \epsilon_n \le \norm{\theta - \theta^0} < (l+1) n \epsilon_n \big\}$. Based on $A_{ij} \stackrel{\text{ind}}\sim \mbox{Bernoulli}(\theta_{ij})$ for $i, j = 1, \ldots, n$, consider testing 
$H_0: \theta = \theta^0 \, \mr{versus} \, H_1: \theta \in U_{l, n}$.
There exists a test function $\Phi_{l, n}$ such that 
\begin{align}\label{eq:pt_vs_comp}
\bbE_{\theta^0} (\Phi_{l, n})  \le \exp(- C_1 l^2 n^2 \epsilon_n^2),  \quad 
\sup_{\theta \in U_{l, n} } \bbE_{\theta} (1- \Phi_{l, n})  \le \exp(- C_2 l^2 n^2 \epsilon_n^2 ),
\end{align}
for constants $C_1, C_2 > 0$ independent of $n$. 
\end{lemma}
\begin{proof}
Since $\theta^0 \in \Theta_k$, there exists $z^0 \in \m Z_{n,k}$ and $Q^0 \in [0, 1]^{k \times k}$ with $\theta^0 = \theta^{z^0, Q^0}$. For $z \in \m Z_{n, k}$, define $U_{l,n}(z) = U_{l,n} \cap \Theta_k(z)$, where $\Theta_k(z)$ is as in \eqref{eq:theta_kz}. Clearly, 
\begin{align}\label{eq:u_z}
U_{l,n}(z) = \big\{\theta^{z,Q} : Q \in [0, 1]^{k \times k}, \, l n \epsilon_n \le \norm{\theta^{z, Q} - \theta^{z^0, Q^0}} < (l+1) n \epsilon_n  \big\},
\end{align}
and $U_{l, n} \subset \cup_{z \in \m Z_{n, k}} U_{l,n}(z)$. We first use Lemma \ref{lem:test} to construct tests against $U_{l,n}(z)$ for fixed $z$. Our desired test is obtained by taking the maximum of all such test functions. 

Fix $z \in \mathcal{Z}_{n, k}$.  Let $\mathcal{N}_{l,n}(z) = \{\theta_{l, n, h} \in U_{l, n}(z): h \in I_{l,n}(z) \}$ be a {\em maximal} $l n \epsilon_n/2$-separated set inside  $U_{l, n}(z)$ for some index set $I_{l,n}(z)$;  
i.e., $\m N_{l,n}(z)$ is such that $\norm{\theta^1 - \theta^2} \geq l n \epsilon_n/2$ for all $\theta^1 \ne \theta^2 \in \mathcal{N}_{l,n}(z)$, 
and no subset of $U_{l, n}(z)$ containing $\m N_{l,n}(z)$ has this property.  We provide a volume argument to determine an upper bound for $\abs{I_{l,n}(z)}$, the cardinality of $\m N_{l, n}(z)$. The separation property implies that Euclidean balls of radius $l n \epsilon_n/4$ centered at the points in $\m N_{l, n}(z)$ are disjoint. Since $B_h^+ := \big\{\theta^{z, Q} : Q \in \mb R^{k \times k}, \norm{\theta^{z, Q} - \theta_{l, n, h}} < l n \epsilon_n/4 \big\}$ is contained inside an Euclidean ball of radius $l n \epsilon_n/4$ centered at $\theta_{l,n,h}$, the sets $B_h^+$ are disjoint as $h$ varies over $I_{l, n}(z)$. By the triangle inequality, all $B_h^+$s lie inside $B^+ = \big\{\theta^{z, Q} : Q \in \mb R^{k \times k}, \norm{\theta^{z,Q} - \theta^0} \le (5l/4+1) n \epsilon_n \big\}$, since $\norm{\theta^{z,Q} - \theta^0} \le \norm{\theta^{z,Q} - \theta_{l,n,h}} + \norm{\theta_{l,n,h} - \theta^0} \le (l+1) n \epsilon_n + l n \epsilon_n/4$. 

It should be noted that the sets $B_h^+$s and $B^+$ are constructed in a way that $Q$ is not restricted to be inside $[0, 1]^{k \times k}$. This allows us to invoke Corollary \ref{cor:ellip1} to identify $B_h^+$ and $B^+$ with appropriate ellipsoids in $\mb R^{k^2}$ and simplify volume calculations. First, since $\theta_{l,n,h} \in \Theta_k(z)$ for each $h$, it follows from (the equality part of) Corollary \ref{cor:ellip1} that $B_h^+ = \{\theta^{z,Q} : Q \in \xi_{k^2}(\bar{Q}_h, \wt{W})\}$ with $\bar{Q}_h$ constructed as in the proof of Lemma \ref{lem:ball} and $\wt{W}_{rs} = n_r n_s/\{(l n \epsilon_n)^2\}$. The equality is crucially used below; also note that $\wt{W}$ does not depend on $h$. Invoking Corollary \ref{cor:ellip1} one more time, we obtain $B^+ \subset \{\theta^{z,Q} : Q \in \xi_{k^2}(\bar{Q}^0, W)\}$, with $W_{rs} = n_r n_s/[\{ (5l/4+1) n \epsilon_n \}^2]$. 
We conclude that the Euclidean ellipsoids $\xi_{k^2}(\bar{Q}_h, \wt{W})$ are disjoint as $h$ varies over $I_{l, n}(z)$ and all of them are contained in $\xi_{k^2}(\bar{Q}^0, W)$. Comparing volumes, 
$$
| \xi_{k^2}(\bar{Q}_h, \wt{W}) | | I_{l,n}(z)| \le | \xi_{k^2}(\bar{Q}^0, W) |.
$$
Using the volume formula in \eqref{eq:vol}
and canceling out common terms, we finally have
\begin{align}\label{eq:card_ub}
| I_{l,n}(z) | \le \bigg\{ \frac {(5l/4+1)}{l/2} \bigg\}^{k^2} \le 9^{k^2}.
\end{align}
We are now in a position to construct the test. The maximality of $\m N_{l,n}(z)$ implies that $\m N_{l,n}(z)$ is an $l n \epsilon_n/2$-net of $U_{l,n}(z)$, i.e., the sets $E_{l, n, z, h} = \{ \theta \in [0, 1]^{n \times n} : \norm{\theta - \theta_{l, n, h}} < l n \epsilon_n/2\}$ cover $U_{l,n}(z)$ as $h$ varies. 
For each $\theta_{l,n,h} \in \m N_{l,n}(z)$, consider testing $H_0: \theta = \theta^0$ versus $H_1: \theta \in E_{l, n, z, h}$ using the test function from Lemma \ref{lem:test}. Lemma \ref{lem:test} is applicable since $\norm{\theta^0 - \theta_{l,n,h}} \ge l n \epsilon_n$; let $\Phi_{l, n, z, h}$ denote the corresponding test with type-I and II errors bounded above by $e^{-C l^2 n^2 \epsilon_n^2}$. Define $\Phi_{l,n} = \max_{z \in \m Z_{n,k}} \max_{h \in I_{l,n}(z)} \Phi_{l,n,z,h}$. For any $\theta \in U_{l,n}$, there exists $z \in \m Z_{n,k}$ and $h \in I_{l,n}(z)$ such that $\theta \in E_{l,n,z,h}$, so that $\bbE_{\theta} (1 - \Phi_{l,n}) \le \bbE_{\theta}(1 - \Phi_{l,n,z,h}) \le e^{-C l^2 n^2 \epsilon_n^2}$. Taking supremum over $\theta \in U_{l,n}$ delivers the desired type-II error. Further, the type-I error of $\Phi_{l,n}$ can be bounded as
\begin{align}
\bbE_{\theta^0} (\Phi_{l,n}) \le  \sum_{z \in \m Z_{n,k}} \sum_{h \in I_{l,n}(z)} \bbE_{\theta^0} (\Phi_{l,n,z,h}) \le k^n 9^{k^2} e^{- C l^2 n^2 \epsilon_n^2},
\end{align}
since $|\m Z_{n,k}| = k^n$ and by \eqref{eq:card_ub}, $|I_{l,n}(z)| \le 9^{k^2}$ for all $z$. The conclusion follows since $n^2 \epsilon_n^2 = k^2\log(n/k) + n \log k \gtrsim k^2 + n \log k$. 

\end{proof}


\section{Proof of Theorem \ref{thm:sbm_nonadaptive}}

\begin{proof}
Let $\bbE_0 \slash \bbP_0$ denote an abbreviation to $\bbE_{\theta^0} \slash \bbP_{\theta^0}$. Since $\theta^0 \in \Theta_k$, there exists $z^0 \in \m Z_{n,k}$ and $Q^0 \in [0, 1]^{k \times k}$ with $\theta^0 = \theta^{z^0, Q^0}$. Recall $\epsilon_n = k^2 \log n/n^2 + \log k/n$ and define $U_n = \big\{\theta \in \Theta_k: \norm{\theta - \theta^0}^2 > M^2 n^2 \epsilon_n^2\big\}$ for some large constant $M > 0$ to be chosen later. Letting $f_{\theta_{ij}}(A_{ij}) = \theta_{ij}^{A_{ij}} (1 - \theta_{ij})^{A_{ij}}$ denote the $\mbox{Bernoulli}(\theta_{ij})$ likelihood evaluated at $A_{ij}$, the posterior probability assigned to $U_n$ can be written as 
\begin{align}\label{eq:post_Un}
\Pi_n(U_n \mid A) = \frac{ \int_{U_n} \prod_{i=1}^n \prod_{j=1}^n \frac{ f_{\theta_{ij}}(A_{ij}) }{ f_{\theta_{ij}^0}(A_{ij}) } \, p(dz, dQ) }{ \int_{\Theta_k} \prod_{i=1}^n \prod_{j=1}^n \frac{ f_{\theta_{ij}}(A_{ij}) }{ f_{\theta_{ij}^0}(A_{ij}) } \, p(dz, dQ) } = \frac{\m N_n}{\m D_n},
\end{align}
where $\m N_n$ and $\m D_n$ respectively denote the numerator and denominator of the fraction in \eqref{eq:post_Un}. Let $\m F_n$ denote the $\sigma$-field generated by $\tilde{A} = (\tilde{A}_{ij})$, with $\tilde{A}_{ij}$ independently distributed as $\mbox{Bernoulli}(\theta_{ij}^0)$; the true data generating distribution. We first claim that there exists a set $\m A_n \in \m F_n$ where we can bound $\m D_n$ from below with large probability under $\bbP_0$ in Lemma \ref{lem:denom}. A proof can be adapted from Lemma 10 of \cite{ghosal2007convergence} and hence omitted.
\begin{lemma}\label{lem:denom}
Assume $\theta^0$ satisfies the conditions of Theorem \ref{thm:sbm_nonadaptive}. Then, there exists a set $\m A_n$ in the $\sigma$-field $\m F_n$ with $\lim_{n \to \infty} \bbP_0(\m A_n) = 1$ such that within $\m A_n$,
\begin{align*}
\m D_n \ge e^{-C n^2 \epsilon_n^2} \Pi\big( \norm{\theta - \theta^0}^2 < n^2 \epsilon_n^2 \big).
\end{align*}
\end{lemma}
In view of Lemma \ref{lem:denom}, it is sufficient to prove that
$$
\lim_{n \to \infty} \bbE_0 \big\{ \Pi_n(U_n \mid A) \ind_{\m A_n^c} \big\} = 0.
$$
For $l \ge M$, let $U_{l, n} = \big\{ \theta \in \Theta_k : l^2 n^2 \epsilon_n^2 \le \norm{\theta - \theta^0}^2 < (l+1)^2 n^2 \epsilon_n^2 \big\}$ denote an annulus in $\Theta_k$ centered at $\theta^0$ with inner and outer Euclidean radii $l n \epsilon_n$ and $(l+1) n \epsilon_n$ respectively. Using a standard testing argument (see, for example, the proof of Proposition 5.1 in \cite{castilloneedles}) in conjunction with Lemma \ref{lem:denom}, one arrives at
\begin{align}\label{eq:castillo_step}
\bbE_0 \big\{ \Pi_n(U_n \mid A) \ind_{\m A_n^c} \big\} \le 
\sum_{l=M}^{\infty} \bigg\{ \bbE_0 (\Phi_{l, n}) + \beta_{l, n} \sup_{\theta \in U_{l, n} } \bbE_{\theta} (1-  \Phi_{l,n}) \bigg\}
\end{align}
where 
\begin{eqnarray}\label{eq:beta_ln}
\beta_{l, n}  = \frac{\Pi(U_{l, n})}{e^{-C n^2 \epsilon_n^2} \Pi( \norm{\theta  - \theta^0}^2  < n^2 \epsilon_n^2) }
\end{eqnarray}
and $\Phi_{l, n}$ is the test function constructed in Lemma \ref{lem:test_ann} for testing $H_0: \theta = \theta^0$ versus $H_1: \theta \in U_{l, n}$ with error rates as in \eqref{eq:pt_vs_comp}. Recall $U_{l,n}(z) = U_{l,n} \cap \Theta_k(z)$ and its equivalent representation in \eqref{eq:u_z} from the proof of Lemma \ref{lem:test_ann}. Since $U_{l,n} \subseteq \cup_{z \in \m Z_{n,k}} U_{l,n}(z)$, from \eqref{eq:prior_U},
\begin{align}\label{eq:num_1}
\Pi(U_{l,n}) \le \sum_{z \in \m Z_{n,k}} \Pi \big\{U_{l,n}(z) \big\} = \sum_{z \in \m Z_{n,k}} \Pi \big\{U_{l,n}(z) \mid z \big\} \, p(z),
\end{align}
where $p(z)$ is the prior probability \eqref{eq:dir_mult} of $z$ under the Dirichlet-multinomial prior. By an application of Lemma \ref{lem:ball}, $U_{l,n}(z) \subset \{\theta^{z, Q} : Q \in \xi_{k^2}(\bar{Q}^0, W) \cap [0, 1]^{k \times k}\}$ with $W_{rs} = n_r n_s/\{(l+1) n \epsilon_n\}^2, 1\le r, s \le k$. Therefore, $\Pi\{ U_{l,n}(z) \mid z \}$ is bounded above by the probability of the set $\xi_{k^2}(\bar{Q}^0, W) \cap [0, 1]^{k \times k}$ under the uniform prior on $Q$, which in turn can be bounded above by the Euclidean volume of $\xi_{k^2}(\bar{Q}^0, W)$. Using volume formula \eqref{eq:vol},  
\begin{align}\label{eq:num_2}
 \Pi \big\{U_{l,n}(z) \mid z \big\} \le | \xi_{k^2}(\bar{Q}^0, W) | = \frac{\pi^{k^2}}{\Gamma(k^2/2+1)} \prod_{r=1}^k \prod_{s=1}^k \frac{(l+1)n \epsilon_n}{n_r n_s}.
\end{align}
Next, consider the term $\Pi( \norm{\theta  - \theta^0}^2  < n^2 \epsilon_n^2)$ in the denominator of the expression for $\beta_{l,n}$. Bound $\Pi( \norm{\theta  - \theta^0}^2  < n^2 \epsilon_n^2) \ge \Pi( \norm{\theta  - \theta^0}^2  < n^2 \epsilon_n^2 \mid z = z^0) p(z^0)$ and using Lemma \ref{lem:ball} once again, 
\begin{align}\label{eq:denom_1}
\Pi\bigg( \norm{\theta  - \theta^0}^2  < n^2 \epsilon_n^2 \mid z = z^0 \bigg) = p\bigg\{ Q : \sum_{r=1}^k \sum_{s=1}^k n_{0r} n_{0s} (Q_{rs} - Q^0_{rs})^2 < n^2 \epsilon_n^2 \bigg\}.
\end{align}
The probability in the right hand side of the above display is the volume of the {\em intersection} of an ellipsoid with $[0, 1]^{k \times k}$, and therefore we cannot simply replace the probability by the volume of the ellipsoid. Instead, we embed an appropriate rectangle inside the intersection of the ellipsoid and $[0, 1]^{k \times k}$. We claim that 
{\small
\begin{align}\label{eq:denom_2}
\prod_{r=1}^k \prod_{s=1}^k [Q_{rs}^0 - \epsilon_n/2, Q_{rs}^0 + \epsilon_n/2] \subset \bigg\{ Q \in [0, 1]^{k \times k}: \sum_{r=1}^k \sum_{s=1}^k n_{0r} n_{0s} (Q_{rs} - Q^0_{rs})^2 < n^2 \epsilon_n^2 \bigg\}.
\end{align}}
First, based on our assumption that all entries of $Q^0$ are bounded away from $0$ and $1$ and the fact that $\epsilon_n \to 0$, it is immediate that the rectangle is contained in $[0, 1]^{k \times k}$ for sufficiently large $n$. Second, for any $Q$ with $|Q_{rs} - Q^0_{rs}| \le \epsilon_n/2$ for all $1 \le r, s \le k$, we have
$$
\sum_{r=1}^k \sum_{s=1}^k n_{0r} n_{0s} (Q_{rs} - Q^0_{rs})^2 \le  \frac{\epsilon_n^2}{4} \sum_{r=1}^k \sum_{s=1}^k n_{0r} n_{0s} = \frac{n^2 \epsilon_n^2}{4},
$$
thereby proving the claim in \eqref{eq:denom_2}. Now we can bound $\Pi( \norm{\theta  - \theta^0}^2  < n^2 \epsilon_n^2 \mid z = z^0)$ from below by the volume of the rectangle, which equals $\epsilon_n^{k^2}$. 
Using this fact along with the bounds \eqref{eq:num_1}, \eqref{eq:num_2} and \eqref{eq:denom_1}, we have from \eqref{eq:beta_ln} that 
\begin{align}\label{eq:beta_ln1}
e^{-C n^2 \epsilon_n^2} \beta_{l,n} \le \frac{\pi^{k^2} (l+1)^{k^2} n^{k^2}}{\Gamma(k^2/2+1)}  \sum_{z \in \m Z_{n,k}} \frac{p(z)}{p(z^0)}.
\end{align}
Since $n_{0r} \ge 1$ for all $r = 1, \ldots, k$, invoke Lemma \ref{lem:prior_rat} to bound $\sum_{z \in \m Z_{n,k}} p(z)/p(z_0) \le | \m Z_{n,k} | e^{ C n \log k} \le e^{C n \log k}$, since $|\m Z_{n,k}| = k^n$. Next, use the well-known fact (see, for example, \cite{abramowitz1964handbook}) that for any $\alpha > 0$, $\Gamma(\alpha + 1) \ge \sqrt{2 \pi} e^{-\alpha} \alpha^{\alpha + 1/2}$ to obtain
\begin{align}\label{eq:beta_ln2}
e^{-C n^2 \epsilon_n^2} \beta_{l,n} \lesssim \{\pi \sqrt{2e} (l+1)\}^{k^2} \bigg(\frac{n}{k}\bigg)^{k^2} \lesssim \{\pi \sqrt{2e}  (l+1)\}^{k^2} e^{k^2 \log(n/k)}.
\end{align}
Substituting in \eqref{eq:castillo_step}, the expression in \eqref{eq:castillo_step} converges to zero for all $M$ larger than a suitable constant.

\end{proof}

\section{Discussion}
In this article, we presented a theoretical investigation of posterior contraction in stochastic block models.  One crucial assumption in our 
current results is that the true number of clusters $k$ is known.  An interesting direction
is to develop a fully Bayesian approach by placing a prior on $k$ and to show that the corresponding procedure yields 
optimal rates of posterior convergence adaptively for all values of $k \in \{1, 2, \ldots, n\}$.  
Such an approach can be connected to nonparametric estimation of networks \cite{bickel2009nonparametric} where one typically assumes a more flexible way of 
data generation;  $A_{ij} \mid \xi_i, \xi_j \sim \mbox{Bernoulli}\{f(\xi_i, \xi_j)\}$, where $f$ is a function 
from $[0,1]^2 \to [0,1]$, called a {\em graphon} and $\xi_i$s are iid random variables on $[0, 1]$. It is well known (refer, for example, to \cite{airoldi2013stochastic}) that
one can approximate a sufficiently smooth graphon using elements of $\Theta_k$. When the smoothness of the graphon is unknown, the prior on $k$ should facilitate the posterior to concentrate in the appropriate region. Using such approximation results and modifying our Theorem 
 \ref{thm:sbm_nonadaptive}, it may be possible to derive posterior convergence rates for estimating a graphon.

%


\bibliographystyle{biometrika}
\bibliography{all_refs}
\end{document}